\documentclass[12points]{amsart}
\usepackage{amsfonts}
\usepackage{amssymb}
\usepackage{graphicx}
\usepackage{amsmath}

\newtheorem{theorem}{Theorem}
\newtheorem{lemma}{Lemma}

\theoremstyle{definition}
\newtheorem{definition}{Definition}
\newtheorem{example}{Example}

\theoremstyle{remark}

\numberwithin{equation}{section}

\everymath{\displaystyle}

\begin{document}

\title{On the simplicity of  Lie algebra  of  Leavitt path algebra}

\author{Adel Alahmedi}
\address{Department of Mathematics, King Abdulaziz University, P.O.Box 80203, Jeddah, 21589, Saudi Arabia}
\curraddr{Department of Mathematics, King Abdulaziz University, P.O.Box 80203, Jeddah, 21589, Saudi Arabia}
\email{adelnife2@yahoo.com}

\author{Hamed Alsulami}
\address{Department of Mathematics, King Abdulaziz University, P.O.Box 80203, Jeddah, 21589, Saudi Arabia}
\email{hhaalsalmi@kau.edu.sa}

%    General info
%\subjclass[2000]{Primary 54C40, 14E20; Secondary 46E25, 20C20}

%\date{January 1, 2001 and, in revised form, June 22, 2001.}



\begin{abstract}
For a field $F$ and a row-finite directed graph $\Gamma$ let  $L(\Gamma)$ be the Leavitt path algebra. We find necessary and sufficient conditions for the Lie algebra
  $[L(\Gamma),L(\Gamma)]$ to be simple.
\end{abstract}

\maketitle

\section{Introduction.}

In [3] G. Abrams and Z. Mesyan found necessary and sufficient conditions for a simple Leavitt path algebra $L(\Gamma)$ to give rise to a simple Lie algebra $[L(\Gamma),L(\Gamma)].$
This result is based on a simple easily checkable criterion for a linear combination of vertices $\sum_{i} \alpha_i v_i, \alpha_i\in F, v_i\in V,$ to lie in $[L(\Gamma),L(\Gamma)].$ In this paper
we extend the result of G.  Abrams and Z. Mesyan to not necessarily simple algebras and find    the necessary and sufficient conditions for a Lie algebra $[L(\Gamma),L(\Gamma)]$ to be simple.

\section{\protect\bigskip Definitions and Terminology}

A (directed) graph $\Gamma =(V,E,s,r)$ consists of two sets $V$ and $E$ that are respectively called vertices and edges, and two maps $s,$ $r:E\rightarrow V$.The vertices $s(e)$ and $r(e)$ are referred to as the source and the range of the edge $e$, respectively. The graph is called row-finite if for all vertices $v\in V$,$card(s^{-1}(v))<\infty .$ A vertex $v$ for which $s^{-1}(v)=\emptyset$ is called a sink. A vertex $v$ such that $r^{-1}(v)=\emptyset$ is called a source. A path $p=e_{1}.....e_{n}$ in a graph $\Gamma $ is a sequence of edges $e_{1}.....e_{n}$ such that $r(e_{i})=s(e_{i+1})$ for $i=1,...,n-1.$ In this case we say that the path $p$ starts at the vertex $s(e_{1})$ and ends at the vertex $r(e_{n}).$ If $s(e_{1})=$ $r(e_{n}),$ then the path is closed. If $p=e_{1}.....e_{n}$ is a closed path and the vertices $s(e_{1})$, $....,$ $s(e_{n})$ are distinct, then the subgraph $($ $s(e_{1}),...,$ $s(e_{n});e_{1},...,e_{n})$ of the graph $\Gamma $ is called a cycle. A cycle of length $1$ is called a loop.

\begin{definition}\label{def1}
Let $W$ be a subset of $V$. We say that
\begin{itemize}
  \item  $W$ is hereditary if $v \in W$ implies $w\in W$ for every vertex $w$ connects to $v.$
  \item$W$ is saturated if $\{r(e) : s(e) = v\}\subseteq W$ implies that $v\in W,$ for every non-sink vertex $v\in V.$
\end{itemize}
\end{definition}

\begin{definition}\label{def2}
We  call an edge $e\in E$ a \textit{fiber} if $s(e)$ is source, $r(e)$ is  sink and $E(V, r(e))=\{e\}.$
\end{definition}

%\begin{definition}\label{def2}
%We  call a connected graph $\Gamma$  a \textit{fork} if one vertex in $V$ is a source, whereas all  other vertices are sinks.
%\end{definition}

\begin{definition}\label{def3}
We call a vertex $v$ in a connected graph $\Gamma(V,E)$ a \textit{balloon} over a nonempty subset $W$ of $V$ if (i) $v\notin W,$ (ii) there is a loop $C\in E(v,v),$ (iii) $E(v,W)\neq\emptyset,$ (iv) $E(v,V)=\{C\}\cup E(v,W),$ and (v) $E(V,v)=\{C\}.$
\end{definition}

Let $\Gamma $ be a row-finite graph and let $F$ be a field. The Leavitt path $F$-algebra $L(\Gamma )$ is the $F$-algebra presented by the set of generators $\{v | v\in V\},$ $\{e,$ $e^{\ast }|$ $e\in E\}$ and the set of relators (1) $v_{i}v_{j}=\delta _{v_{i},v_{j}}v_{i}$ for all $v_{i},v_{j}\in V;$ (2) $s(e)e=er(e)=e,$ $r(e)e^{\ast }=e^{\ast }s(e)=e^{\ast}$ for all $e\in E;$ (3) $e^{\ast }f=\delta _{e,f}r(e),$ for all $e,$ $f\in E;$ (4) $v=\sum_{s(e)=v}ee^{\ast }$, for an arbitrary vertex $v$ which is not a sink. The mapping which sends $v$ to $v$ for $v\in V,$ $e$ to $e^{\ast }$ and $e^{\ast }$ to $e$ for $e\in E,$ extends to an involution of the algebra $L(\Gamma ).$ If $p=e_{1}.....e_{n}$ is a path, then
$p^{\ast}=e_{n}^{\ast }....e_{1}^{\ast }.$ In what follows we consider only row-finite directed graphs. We call a graph $\Gamma$ simple if the Leavitt path algebra $L(\Gamma)$ is simple.
The conditions for a graph to be simple are given in [1].

%\begin{definition}\label{def4}
%We call a graph $\Gamma$ almost simple if $\Gamma$ contains subgraphs $\Gamma\supset\Gamma_1\supset\Gamma_2$ where $\Gamma_2$ is a simple; $\Gamma_1$ is obtained from $\Gamma_2$ by adding balloons; $\Gamma$ is obtained from $\Gamma_1$ by adding fibers.
%\end{definition}

Let $A$ be  an associative $F-$algebra. For elements $a,b\in A$, let $[a,b]=ab-ba$ be their the commutator. Then $A^{(-)}=(A, [,])$ is a Lie algebra.
%The space of $*-$skew-symmetric elements $K=K(L(\Gamma),*)=\{a\in A| a^*=-a\}$ is a subalgebra of the Lie algebra $L(\Gamma)^{(-)}.$
If $A$ is an associative algebra and $S$ is a subset of $ A,$ we will denote the ideal of $A$ generated by $S$ as $id_{A}(S).$

\section{ Lie algebra of Leavitt path algebra  }

We start with theorem by G. Abrams and Z. Mesyan in [3].
\begin{theorem}([3])\label{th1}
Let $\Gamma(V,E)$ be a directed graph. Let $L(\Gamma)$ be a simple algebra.
\begin{itemize}
  \item[(i)] If $V$ is infinite then the Lie algebra $[L(\Gamma),L(\Gamma)]$ is simple;
  \item[(ii)]  If $V$ is finite, then $[L(\Gamma),L(\Gamma)]$ is simple if and only if $1_{L(\Gamma)}=\sum_{v\in V}v\notin [L(\Gamma),L(\Gamma)].$
\end{itemize}
\end{theorem}

%\begin{proof}
%see [2] for the proof.
%\end{proof}

There exist however non-simple Leavitt path algebras having the Lie algebra $[L(\Gamma),L(\Gamma)]$ simple.

\begin{example}
Let $\Gamma=$\includegraphics[width=.15\textwidth]{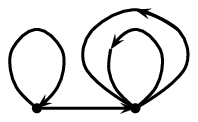}. The Lie algebra $[L(\Gamma),L(\Gamma)]$ is isomorphic to the
Lie algebra of infinite finitary matrices over the Leavitt algebra $L(2)$ and therefore is simple.
\end{example}
%% give definition of simple graph and hereditary and saturated

The following theorem gives a classification of directed graph having $[L(\Gamma),L(\Gamma)]$ simple.

\begin{theorem}\label{th2}
Let $\Gamma(V,E)$ be a directed row-finite graph. The Lie algebra $[L(\Gamma),L(\Gamma)]$ is simple if and only if either $L(\Gamma)$ is
simple- this case is covered by Theorem\ref{th1} - or $\Gamma$ contains a simple subgraph $W$ such that every point $v \in V\setminus W$ is
a balloon over $W,$ and $\sum_{w\in r(E(v,W))} w\in [L(W),L(W)].$
\end{theorem}

 We will prove the theorem by proving a series of lemmas. The first lemma is due to  G.  Abrams and Z. Mesyan, [3].
 We will state it without proof.

\begin{lemma}([3])\label{le1}
Let $\Gamma(V,E)$ be a directed graph. Then $[L(\Gamma),L(\Gamma)]=(0)$ if and only if $\Gamma$ is a disjoint union of  \includegraphics[width=.04\textwidth]{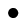},\includegraphics[width=.05\textwidth]{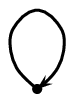}( vertices and loops).
\end{lemma}

\begin{lemma}\label{le2}
Let $\Gamma(V,E)$ be a row-finite graph. If the Lie algebra $[L(\Gamma),L(\Gamma)]$ is nonzero simple, then every cycle  has an exit.
\end{lemma}

\begin{proof}
%Suppose, on the contrary, that $\Gamma$ has a no exist cycle $C$ of length $d\geq1.$ Since $L(C)\cong M_d(F[t,t^{-1}]).$ let $a$ be sum of all vertices on the
%cycle $C.$ Then $a$ is the identity of $L(C)$ and $L(C)=aL(\Gamma)a.$
Let $C$ be a no exist cycle of $\Gamma$ of length $d.$ Then $L(C)\cong M_d(F[t,t^{-1}]).$ Let $a$ be the sum of all vertices on the cycle $C.$ The element
$a$ is the identity of $L(C)$ and $L(C)=aL(\Gamma)a.$
Consider the ideal $J_n=(1-t)^n F[t,t^{-1}]$ of $F[t,t^{-1}].$ Now, if $d\geq 2,$ then $[M_d(J_n),M_d(J_n)]\neq (0)$ for all $n\geq 1,$ see [5].
Let $I_n=id_{L(\Gamma)}(M_d(J_n)).$ Then $[I_n,I_n]\triangleleft  [L(\Gamma),L(\Gamma)]$ and because of simplicity of $[L(\Gamma),L(\Gamma)]$
we have $[L(\Gamma),L(\Gamma)]=[I_n,I_n]\subseteq I_n.$ Hence $[L(\Gamma),L(\Gamma)]\cap L(C)\subseteq I_n\cap L(C)=M_d(J_n).$ Since $\cap_n J_n=(0),$
it follows that $[L(\Gamma),L(\Gamma)]\cap L(C)=(0),$ but $(0)\neq [M_d(F[t,t^{-1}]), M_d(F[t,t^{-1}])]\subseteq [L(\Gamma),L(\Gamma)]\cap L(C).$
A contradiction. Hence $d=1.$ Thus $C$ is a loop. Since $C$ has no exit and can not be isolated there exist an edge $e\in E,$ such that $s(e)\notin V(C)=\{v\}.$
Let $J_n=(v-C)^n L(C),$ $I_n=id_{L(\Gamma)}(J_n),$ $vI_nv\subseteq J_n. $ Now, $[eJ_n,J_n]=eJ_n\neq (0).$ Hence $[I_n,I_n]\neq (0),$
$[L(\Gamma),L(\Gamma)]=[I_n,I_n]\subseteq I_n$ and therefore $v[L(\Gamma),L(\Gamma)]v\subseteq J_n.$ Since $\cap_n J_n=(0)$ it follows that
$v[L(\Gamma),L(\Gamma)]v=(0),$ but $[e^*,e]=v-ee^*,$ and $v[e^*,e]v=v\neq 0.$ A contradiction.
\end{proof}

The algebra $L(\Gamma)$ is graded: $dg(v)=0,$ $dg(e)=1,$ $dg(e^*)=-1$ for all $v\in V, e\in E.$ In [9] it is shown that every graded ideal $I$ of $L(\Gamma)$
is generated (as an ideal ) by $I\cap V.$ Thus there is a one-to-one correspondence between graded ideals and hereditary saturated subsets of $V.$

\begin{lemma}\label{le3}
Let $\Gamma(V,E)$ be a row-finite graph. Let $W$ be nonempty hereditary and saturated  subset of $V.$
Let $I=id_{L(\Gamma)}(W).$  If $[L(\Gamma),L(\Gamma)]$ is nonzero simple, then $[I,I]\neq (0).$
\end{lemma}

\begin{proof}
If $[I,I]=(0),$ then, in particular, $[L(W),L(W)]=(0),$ $W$ is a disjoint union of  \includegraphics[width=.04\textwidth]{point}, and \includegraphics[width=.05\textwidth]{loop}.
This implies that for every vertex $w\in W$ there exist an edge $e\in E$ such that $r(e)=w,$ $s(e)\notin W$ otherwise $w$ is isolated in $\Gamma.$
Now, $e,e^*\in I$ and $[e,e^*]\neq 0.$ Lemma is proved.
\end{proof}

\begin{lemma}\label{le4}
Let $\Gamma(V,E)$ be a a row-finite graph. If $[L(\Gamma),L(\Gamma)]$ is nonzero simple, then there exists a minimal hereditary saturated subset in $V.$
\end{lemma}

\begin{proof}
We need to show that the intersection of all nonzero graded ideals in $L(\Gamma)$ is nonzero. If $I$ is a nonzero graded ideal of $L(\Gamma)$ then by Lemma~\ref{le3}
$[I,I]\neq  (0).$ Since $[L(\Gamma),L(\Gamma)]$ is simple, then $[L(\Gamma),L(\Gamma)]=[I,I]$  and therefore $[L(\Gamma),L(\Gamma)]$ lies in the intersection of all
nonzero graded ideals of $L(\Gamma).$
\end{proof}

Let $\Gamma(V,E)$ be a a row-finite graph. Suppose $[L(\Gamma),L(\Gamma)]$ is nonzero simple. Let $W$ be a minimal hereditary saturated subset in $V.$
Let $I=id_{L(\Gamma)}(W),$ $\Gamma'=(V\setminus W, E\setminus E(V, W)).$ We assume that $W\neq V,$ that is $L(\Gamma)$ is not simple.
Since $L(\Gamma')\cong L(\Gamma)/I$ and $[L(\Gamma),L(\Gamma)]\subseteq I$  it follows that $[L(\Gamma'),L(\Gamma')]=(0).$
By Lemma~\ref{le1} $\Gamma'$ is a disjoint union of \includegraphics[width=.04\textwidth]{point}, and \includegraphics[width=.05\textwidth]{loop}.

\begin{lemma}\label{le5}
$\Gamma'$ does not have components  \includegraphics[width=.04\textwidth]{point}.
\end{lemma}

\begin{proof}
Let a vertex $v\in V\setminus W$ be isolated in $\Gamma'.$ Then $E(V\setminus W,v)=\emptyset.$ Since $W$ is hereditary and $v\notin W$ we conclude that
$E(V,v)=\emptyset.$ Since $v$ can not be isolated in $\Gamma$ it can not be a sink, $E(v,V)\neq\emptyset.$ But $E(v,V\setminus W)=\emptyset,$ hence all
descendants of $v$ lie in $W.$ Since $W$ is saturated we conclude that $v\in W,$ a contradiction.
\end{proof}

\begin{lemma}\label{le6}
Every vertex $v\in V\setminus W$ is a balloon over $W.$
\end{lemma}

\begin{proof}
By what we have shown $\Gamma'$ is a disjoint union of loops \includegraphics[width=.05\textwidth]{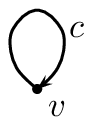}. It is easly to see that $E(V,v)=\{c\}$ and $E(v,V\setminus W)=\{c\}.$
If $E(v,W)=\emptyset$ then the loop  \includegraphics[width=.05\textwidth]{loopl} is isolated in $\Gamma.$ Hence $E(v,W)\neq\emptyset.$
Thus $v$ is a balloon over  $W.$
\end{proof}

Let $S_0$ be the span of all elements $pp^*,$ where $p$ is a path on $\Gamma$ including pathes of length zero(that is vertices).
Let $S_1$ be the span of all elements $pq^*,$ where $p,q$ are pathes on $\Gamma,$ $r(p)=r(q),$ $p\neq q.$ It follows from the description
of a Groebner - Shirshov basis of $L(\Gamma)$ [4] that $L(\Gamma)=S_0+S_1$ is a direct sum of vector spaces.
Let $M$ be the semigroup generated by $V\cup E\cup E^*.$ It is easily to see that (i) $M=(M\cap S_0)\cup (M\cap S_1),$
(ii) for arbitrary elements $a,b\in M$ if $0\neq ab\in S_i,$ then $ba\in S_i$ or $ba=0,$ for $i=0,1.$

\begin{lemma}\label{le7}
$[I,I]\cap S_0=span\{[p,p^*] | \text{ $p$ is a path on $\Gamma$ , $r(p)\in W$}\}.$
\end{lemma}

\begin{proof}
The ideal $I$ is spanned by elements $pq^*;$ $p,q$ are paths, $r(p)=r(q)\in W.$ Consider two such elements $p_1q_1^*$ and $p_2q_2^*,$
$0\neq p_1q_1^*p_2q_2^*\in S_0.$ Since $q_1^*p_2\neq 0$ it follows that $p_2=q_1u$ or $q_1=p_2u,$ where $u$ is a path on $\Gamma.$
Consider the first case, $p_2=q_1u.$ Then $p_1q_1^*p_2q_2^*=p_1uq^*_2.$ Since this element lies in $S_0$ we conclude that $q_2=p_1u.$
Now, $p_2q_2^*p_1q_1^*=q_1uu^*p_1^*p_1q_1^*=(q_1u)(q_1u)^*$ and therefore $[p_1q_1^*, p_2q_2^*]=(p_1u)(p_1u)^*-(q_1u)(q_1u)^*=[p_1u,(p_1u)^*]-[q_1u,(q_1u)^*].$
Remember that $r(u)=r(q_2)\in W.$ Let $q_1=p_2u.$ Then $p_1q_1^*p_2q_2^*=p_1u^*p_2^*p_2q_2^*=p_1(q_2u)^*.$ Again $p_1q_1^*p_2q_2^*\in S_0$ implies $p_1=q_2u.$
Now, $p_2q_2^*p_1q_1^*=p_2q_2^*q_2uu^*p_2^*=(p_2u)(p_2u)^*.$ Therefore, $[p_1q_1^*,p_2q_2^*]=(q_2u)(q_2u)^*-(p_2u)(p_2u)^*=[q_2u,(q_2u)^*]-[p_2u,(p_2u)^*]$ and $r(u)=r(p_1)\in W.$
\end{proof}

Let $v\in V\setminus W,$ $E(v,W)=\{e_1,\ldots, e_n\},$ $r(e_i)=w_i$ for $1\leq i\leq n.$ Let $w=\sum_{i=1}^n w_i.$

\begin{lemma}\label{le8}
$w\in [L(W),L(W)].$
\end{lemma}

\begin{proof}
Since $v$ is a balloon over $W,$ let $c$ be the loop from $E(v,v)$, we have $v=cc^*+\sum_{i=1}^n e_ie_i^*.$
Hence $c^*c-cc^*=v-(v-\sum_{i=1}^n e_ie_i^*)=\sum_{i=1}^n e_ie_i^*=\sum_{i=1}^n [e_i,e_i^*]+\sum_{i=1}^n e_i^*e_i=\sum_{i=1}^n [e_i,e_i^*]+w.$
Thus $w=c^*c-cc^*-\sum_{i=1}^n [e_i,e_i^*]=[c^*,c]-\sum_{i=1}^n [e_i,e_i^*]\in [L(\Gamma),L(\Gamma)]=[I,I].$ Hence $w\in [I,I]\cap S_0.$
By Lemma~\ref{le7} $w=\sum_{i} \alpha_i [p_i,p_i^*],\, \alpha_i\in F,\, r(p_i)\in W.$ We will distinguish between pathes that start with an edge from
$E(V\setminus W, W)$ and paths that lie entirely on $W,$  $w=\sum_{i} \alpha_{e,i} [ep_{e,i}, p_{e,i}^*e^*]+\sum \beta [q,q^*],$ where $e$ runs over $E(V\setminus W,W),$ $\alpha_{e,i}\in F,$
$p_{e,i}$ and $q$ are paths on $W.$ We have, $w= \sum_{i} \alpha_{e,i} (ep_{e,i} p_{e,i}^*e^*-r(p_{e,i}))+\sum \beta [q,q^*].$ Fix $e\in E(V\setminus W, W).$ From the description of the basis of
$L(\Gamma)$ in [4] it follows that $\sum_{i} \alpha_{e,i} ep_{e,i}p_{e,i}^*e^*=0$ and therefore $\sum_{i} \alpha_{e,i} p_{e,i}p_{e,i}^*=0.$
Now $\sum_{i} \alpha_{e,i} (ep_{e,i} p_{e,i}^*e^*-r(p_{e,i}))=\sum_{i} \alpha_{e,i} [p_{e,i},p_{e,i}^*]\in[L(W),L(W)].$ \newline Hence $w=\sum\alpha_{e,i} [p_{e,i},p_{e,i}^*]+ \sum \beta [q,q^*]\in [L(W),L(W)].$
\end{proof}

We proved Theorem~\ref{th2} in one direction.

\section{ Simplicity of the Lie algebra of Leavitt path algebra}

Let $\Gamma(V,E)$ be a graph. Suppose that $W\subsetneqq V$ is a simple subgraph, every vertex $v\in V\setminus W$ is a balloon over $W$ and
$\sum_{w\in r(E(v,W))}w $ lies in $[L(W),L(W)].$ We will show that the algebra $[L(\Gamma),L(\Gamma)]$ is simple. As above, denote $I=id_{L(\Gamma)}(W).$ The following
lemma was proved in [5].

\begin{lemma}\label{le9}
$I$ is a simple algebra.
\end{lemma}

\begin{lemma}\label{le10}
Let $A$ be an arbitrary simple  algebra with two orthogonal idempotents $e_1,e_2.$ Then $A=[A,A]+e_iAe_i,\, i=1,2.$
\end{lemma}

\begin{proof}
We have $A=Ae_1A.$ For arbitrary elements $a,b\in A,$ $ae_1b=[a, e_1b]+e_1ba.$ Similarly, $A=Ae_2A.$  For arbitrary elements $a,b\in A,$ we have
$e_1ae_2b=[e_1ae_2, e_2b]+e_2be_1ae_2.$ We proved that $A=[A,A]+e_2Ae_2.$ The equality  $A=[A,A]+e_1Ae_1$ is proved similarly.
\end{proof}

\begin{lemma}\label{le11}
$[L(\Gamma),L(\Gamma)]=[I,I]$ .
\end{lemma}

\begin{proof}
We have $L(\Gamma)=I+span\{c_v^n |n\geq 0, v\in V\setminus W\}+span\{(c_v^*)^n |n\geq1,v\in V\setminus W\}.$
Let $w\in W.$ Then, by Lemma \ref{le10}, $I=[I,I]+wIw.$
\newline Hence $[c_v^n,I]=[c_v^n,[I,I]+wIw]=[c_v^n,[I,I]]\subseteq [I,I].$
Similarly, $[(c_v^*)^n,I]\subseteq [I,I].$ It remains to show that $[c_v^n,(c_v^*)^m]\in [I,I].$ Let $c=c_v.$
Suppose at first that $m>n.$ Then
\begin{alignat*}{1}
[c^n,(c^*)^m]&=c^n(c^*)^m-(c^*)^{m-n}\\
&=c^{n-1}(cc^*)(c^*)^{m-1}-(c^*)^{m-n}\\
&=c^{n-1}(v-\sum e_ie_i^*)(c^*)^{m-1}-(c^*)^{m-n}\\
&=(c^{n-1}(c^*)^{m-1}-(c^*)^{m-n})-c^{n-1}\sum e_ie_i^*(c^*)^{m-1}.
\end{alignat*}
The first summand $c^{n-1}(c^*)^{m-1}-(c^*)^{m-n}=[c^{n-1},(c^*)^{m-1}]$ and we can apply the induction assumption. Furthermore, \newline$c^{n-1}e_ie_i^*(c^*)^{m-1}=[c^{n-1}e_i,e_i^*(c^*)^{m-1}]+e_i^*(c^*)^{m-1}c^{n-1}e_i=[c^{n-1}e_i, e_i^*(c^*)^{m-1}],$ since $e_i^*(c^*)^{m-1}c^{n-1}e_i=e_i^*(c^*)^{m-n}e_i=0.$
Now, let $n>m.$ Then
\begin{alignat*}{1}
[c^n,(c^*)^m]&=c^n(c^*)^m-c^{n-m}\\
&=c^{n-1}(v-\sum e_ie_i^*)(c^*)^{m-1}-c^{n-m}\\
&=[c^{n-1},(c^*)^{m-1}]-c^{n-1}\sum e_ie^*_i(c^*)^{m-1}.
\end{alignat*}
As above,
\begin{alignat*}{1}
c^{n-1}e_ie_i^*(c^*)^{m-1}&=[c^{n-1}e_i,e_i^*(c^*)^{m-1}]+e_i^*(c^*)^{m-1}c^{n-1}e_i\\
&=[c^{n-1}e_i,e_i^*(c^*)^{m-1}]+e_i^*c^{n-m}e_i=[c^{n-1}e_i,e_i^*(c^*)^{m-1}].
\end{alignat*}
Finally, let $n=m.$ As above we conclude that
\begin{alignat*}{1}
[c^n,(c^*)^n]&=[c^{n-1},(c^*)^{n-1}]-c^{n-1}\sum e_ie_i^*(c^*)^{n-1},\\
\sum c^{n-1}e_ie_i^*(c^*)^{n-1}&=\sum [c^{n-1}e_i, e_i^*(c^*)^{n-1}]+\sum e_i^*e_i\in [I,I] \text{ by our assumption.}
\end{alignat*}
 Lemma is proved.
\end{proof}

\begin{lemma}\label{le12}
The algebra $[I,I]$ has zero center.
\end{lemma}

\begin{proof}
I. Herstein [7] proved that in a simple associative algebra $A$ of dimension bigger than $4$ over its center, $[A,A]$ generates $A.$ Hence an elements from $I,$
that commutes with $[I,I],$ lies in the center of $I.$ An arbitrary element from $I$ looks as $z=a_0+\sum_{e\in E(v_i,W)}ea_e+ \sum_{e\in E(v_i,W)} b_e^*e^*+\sum_{e\in E(v_i,W)\atop f\in E(v_j,W)}ea_{e,f} f^*,$ $a_0,a_e,b_e, a_{e,f}\in L(W).$ Suppose that $z$ lies in the center of $I.$ Commuting $z$ with idempotents $w\in W,\, ee^*, e\in E(v_i,W)$ we see that $z=a_0+\sum_{e\in E(v_i,W)} ea_e e^*.$ This implies that $a_0$ lies in the center of $W.$ Therefore by [2], $|W|<\infty$ and $a_0=\alpha\sum_{w\in W} w, \, \alpha\in F.$ Multiplying $z$ on the left by $e^*$ and on the right by $e,\, e\in E(v_i,W),$ we get
$r(e)z=a_e=\alpha r(e).$ We proved that $z=\alpha(\sum_{w\in W} w+\sum_{e\in E(v_i,W)} ee^*).$ Now choose a vertex $v_i\in V\setminus W$ and an edge $f\in E(v_i, W).$ We have $zc_if=0,$ whereas $c_ifz=\alpha c_i f.$ Hence $\alpha=0.$ Lemma is proved.

\end{proof}

Now it remans to refer to Herstein's theorem about simplicity of $[I,I]/center,$ see [8]. Hence $[I,I]$ is simple and therefore  $[L(\Gamma),L(\Gamma)]$ is simple.

\section*{Acknowledgement}

The authors would like to thank professor Efim Zelmanov for his constant advise and valuable help during the preparation of this work.
The authors would also like to express their appreciation to professor  S. K. Jain for carefully reading the manuscript and for offering his comments.
This paper was funded by King Abdulaziz University, under grant No. (7-130/1433 HiCi). The authors, therefore, acknowledge technical and financial support of KAU.

\bibliographystyle{amsplain}

\begin{thebibliography}{10}

\bibitem {AA} G. Abrams, G. Aranda Pino,  \textit{The Leavitt path algebra of a graph, } J. Algebra 293 (2005), 319-334

\bibitem {AC} G. Aranda Pino, K. Crow  \textit{The center of a Leavitt path algebra, } Rev. Mat. Iberoamericana Volume 27, Number 2 (2011), 621-644

\bibitem {AM} G. Abrams, Z. Mesyan, \textit{Simple Lie algebra arising from Leavitt path algebra, } Journal of pure and applied algebra, 216(2012), 2303-2313.

\bibitem {AAJZ} A. Alahmadi, H. Alsulami, S.K. Jain, E. Zelmanov, \textit{Leavitt path algebras of finite Gelfand {$–$}Kirillov dimension, } Journal of Algebra and Its Applications, 171(2012)

\bibitem{AlAl} A. Alahmadi, H. Alsulami, \textit{Simplicity of  Lie algebra of skew elements of  Leavitt path algebra,} submitted

\bibitem{C} P. Colak, \textit{Two-sided ideals in Leavitt path algebras,} Journal of Algebra and Its Applications 10-5 (2011)

\bibitem{H} I.N.Herstein, \textit{Topics in Ring Theory, } Mathematics Lecture Notes, University of Chicago,1965

\bibitem{H2} I.N.Herstein, \textit{Rings with Involution, } Mathematics Lecture Notes, University of Chicago,1976.

\bibitem{MT} Mark Tomforde, \textit{ Uniqueness theorems and ideal structure for Leavitt path algebras, } Journal of Algebra 318 (2007), 270–299
\end{thebibliography}

\end{document}